\documentclass[]{amsart}

\usepackage{mathptmx} 
\usepackage{amsmath}
\usepackage{amssymb}
\usepackage{txfonts}
\usepackage{microtype}

\newtheorem{theorem}{Theorem}[section]
\newtheorem{lemma}[theorem]{Lemma}
\newtheorem{definition}[theorem]{Definition}

\begin{document}

\title{A Simple Proof of Vitali's Theorem for Signed Measures}

\author{T.~Samuel}
\address{Fachbereich 3 Mathematik, Universit\"at Bremen, Postfach 330440, Bibliothekstra{\ss}e 1, 28359 Bremen, Germany.}
\email{tony@math.uni-bremen.de}
\date{\today}

\begin{abstract}
There are several theorems named after the Italian mathematician Vitali.  In this note we provide a simple proof of an extension of Vitali's Theorem on the existence of non-measurable sets.  Specifically, we show, without using any decomposition theorems, that there does not exist a non-trivial, atom-less, $\sigma$-additive and translation invariant set function $\mathcal{L}$ from the power set of the real line to the extended real numbers with $\mathcal{L}([0,1]) = 1$.  (Note that $\mathcal{L}$ is not assumed to be non-negative.)
\end{abstract}



\maketitle

\section{Introduction}

A `measure', in its simplest form, can be thought of as a set function from a given collection of subsets of a given space to the extended real numbers, which satisfies a certain number of conditions. Two natural properties that one would like a measure on the power set of the real numbers to satisfy are $\sigma$-additivity (Definition \ref{def:2}) and translation invariance (Definition \ref{def:1}).  Two other natural assumptions are that the unit interval has measure one and the measure of a single point is zero.  In 1905 Vitali \cite{Vit:1905} showed that there cannot exist a non-trivial and non-negative set function on the power set of the real numbers which satisfies these properties.  Interestingly, the result also holds when the condition of non-negativity is removed.  In this note we provide a simple proof of this latter result, without referring to any decomposition theorems such as the Hahn-Jordan Decomposition Theorem (see for instance \cite[Chapter 3]{Bogachev:2007}).

\subsection*{Outline}

In Section \ref{sec:notation} we explain relevant notation and present basic definitions.  Section \ref{sec:Theorems} contains a precise statement of Vitali's Theorem.  Since it is often remarked that this landmark result lead to the birth of modern day measure theory, in Section \ref{sec:examples}, we give a very brief introduction to Borel $\sigma$-algebras and measures via examples.  Finally, in Section \ref{Sec:sectio2} we provide our simple proof of Vitali's Theorem, without assuming the condition of non-negativity.

\section{Notation}\label{sec:notation}

We let $\mathbb{R}$ denote the set of real numbers and let $\mathbb{Q}$ denote the set of rational numbers.  The set of extended real numbers, namely the set $\mathbb{R} \cup \{ \pm \infty \}$, will be denoted by $\overline{\mathbb{R}}$.  Following, standard conventions, we let $[0, 1]$ denote the closed unit interval, we let $(0, 1)$ denote the open unit interval and we let $[0, 1)$ and $(0, 1]$ denote the half open unit intervals.  We denote the \textit{power set} of $\mathbb{R}$ by  $\mathcal{P}(\mathbb{R})$, that is the collection of all subsets of $\mathbb{R}$ including the empty-set $\emptyset$.  Here, one needs to take care, as we do not want to run into a set theoretical paradox, and so, to insure the existence of $\mathcal{P}(\mathbb{R})$ we naturally assume Zermelo-Fraenkel axioms of set theory and, for reasons which will become clear, the axiom of choice (see \cite{Cohen:1966} for a good introduction to set theory).

By a \textit{set function} we mean a function which maps some class of subsets of a given space to $\overline{\mathbb{R}}$; in our case the power set of $\mathbb{R}$.  Two definitions that are required to state Vitali's Theorem are $\sigma$-additivity and translation invariance. 

\begin{definition}\label{def:2}
A set function $\mathcal{L}: \mathcal{P}(\mathbb{R}) \to \overline{\mathbb{R}}$ is called \textit{$\sigma$-additive} if for any collection $\{ A_{n} \}_{n \in \mathcal{I}}$ of pairwise disjoint subsets of $\mathbb{R}$ we have
\begin{align*}
\mathcal{L}\,\bigg(\bigcup_{n \in \mathcal{I}} A_{n}\bigg) = \sum_{n \in \mathcal{I}} \mathcal{L}(A_{n}),
\end{align*}
for some finite or countable index set $\mathcal{I}$.
\end{definition}

\begin{definition}\label{def:1}
A set function $\mathcal{L}: \mathcal{P}(\mathbb{R}) \to \overline{\mathbb{R}}$ is said to be \textit{translation invariant} if $\mathcal{L}(A + t) = \mathcal{L}(A)$, for each $A \in \mathcal{P}(\mathbb{R})$ and every $t \in \mathbb{R}$, where $A + t \coloneqq \{ a + t : a \in A \}$.
\end{definition}

A set function $\mathcal{L}: \mathcal{P}(\mathbb{R}) \to \overline{\mathbb{R}}$ is called \textit{non-negative} if $\mathcal{L}(A) \geq 0$, for all $A \in \mathcal{P}(\mathbb{R})$, and by \textit{non-trivial} we mean that there exists an $A \in \mathcal{P}(\mathbb{R})$, such that $\mathcal{L}(A) \not\in \{0, \pm\infty\}$.  Finally, a set function $\mathcal{L}: \mathcal{P}(\mathbb{R}) \to \overline{\mathbb{R}}$ is said to be \textit{atom-less}, if $\mathcal{L}(\{ x \}) = 0$ for every point $x \in \mathbb{R}$.

\section{Vitali's Theorems}\label{sec:Theorems}

In 1905 Vitali \cite{Vit:1905} proved the following theorem.

\begin{theorem}\label{Main-Theo1}
There does not exist a non-trivial and non-negative set function $\mathcal{L}: \mathcal{P}(\mathbb{R}) \to \overline{\mathbb{R}}$ that satisfies the following.
\begin{enumerate}
\item[(a)] $\mathcal{L}([0, 1]) = 1$.
\item[(b)] If $A, B \in \mathcal{P}(\mathbb{R})$ with $A \subseteq B$, then $\mathcal{L}(A) \leq \mathcal{L}(B)$.
\item[(c)] The set function $\mathcal{L}$ is $\sigma$-additive and translation invariant.
\end{enumerate}
\end{theorem}

Observe that if there exists a set function satisfying the conditions of Theorem \ref{Main-Theo1}, then $\mathcal{L}$ is atom-less.  Since by translation invariance $\mathcal{L}(\{0\}) = \mathcal{L}(\{x\})$, for all $x \in \mathbb{R}$.  Thus, by condition (b) of Theorem \ref{Main-Theo1} and $\sigma$-additivity
\begin{align*}
\mathcal{L}([0,1]) \geq \mathcal{L}(\mathbb{Q} \cap [0, 1]) &= \sum_{q \in \mathbb{Q} \cap [0, 1]}  \mathcal{L}(\{q\}) = \sum_{q \in \mathbb{Q} \cap [0, 1]} \mathcal{L}(\{0\}),
\end{align*}
which can only occur if $\mathcal{L}(\{0\}) = 0$ since by condition (a), $\mathcal{L}([0, 1]) = 1$.  Therefore, a set function satisfying the conditions of Theorem \ref{Main-Theo1} is atom-less.

As mentioned in the introduction, the aim of this note is to provide a simple proof of the following extension, without referring to any technical results, such as the Hahn-Jordan Decomposition Theorem.

\begin{theorem}\label{Main-Theo2}
There does not exist a non-trivial  and atom-less set function $\mathcal{L}: \mathcal{P}(\mathbb{R}) \to \overline{\mathbb{R}}$ which is $\sigma$-additive and translation invariant with $\mathcal{L}([0, 1]) = 1$.
\end{theorem}

Observe that, since the conditions of Theorem \ref{Main-Theo2} are weaker than those of Theorem \ref{Main-Theo1}, Theorem \ref{Main-Theo1} follows as a corollary to Theorem \ref{Main-Theo2}.  However, for the interested reader the standard proof of Theorem \ref{Main-Theo1} can be found in \cite[pages 49-55]{Rudin:1987}.  Also, note that, in Theorem \ref{Main-Theo2}, the condition $\mathcal{L}([0,1]) = 1$ can be replaced by the condition that there exists an uncountable bounded set $B \in \mathcal{P}(\mathbb{R})$ with $\mathcal{L}(B) \not\in \{ 0, \pm\infty \}$.

\section{Examples of Measures}\label{sec:examples}

After the announcement of Vitali's Theorem, it was discovered that a set function satisfying the natural conditions of Theorem \ref{Main-Theo1} exists if one restricts the domain of definition to a suitable subset of $\mathcal{P}(\mathbb{R})$, namely $\mathcal{B}(\mathbb{R})$ the Borel $\sigma$-algebra of $\mathbb{R}$.  This notion of Borel $\sigma$-algebras has formed the foundations of modern day measure theory.

We will now introduce the Borel $\sigma$-algebra for $\mathbb{R}$ and the $1$-dimensional Lebesgue measure $\lambda$.  This measure is often considered to be the most important measures in modern day measure theory.  Further, observe that $\lambda$ is a set function which satisfies the conditions of Theorem \ref{Main-Theo1} with domain $\mathcal{B}(\mathbb{R})$.

\begin{definition}
The smallest subset of the power set of $\mathbb{R}$, which satisfies the following properties, is called the Borel $\sigma$-algebra of $\mathbb{R}$ and is denoted by $\mathcal{B}(\mathbb{R})$.
\begin{enumerate}
\item[(a)] $\mathbb{R} \in \mathcal{B}(\mathbb{R})$.
\item[(b)] If $A \in \mathcal{B}(\mathbb{R})$, then $A^{c} \in \mathcal{B}(\mathbb{R})$; where $A^{c}$ denotes the complement of $A$.
\item[(c)] If $\{ A_{n} \}_{n \in \mathbb{N}}$ is a countable collection of elements of $\mathcal{B}(\mathbb{R})$, then $\bigcup_{n \in \mathbb{N}} A_{n} \in \mathcal{B}(\mathbb{R})$.
\item[(d)] If $x \in \mathbb{R}$ and if $r$ is a non-negative real number, then
\begin{align*}
B(x, r) \coloneqq \left\{ y  \in \mathbb{R} : \lvert x - y \rvert \leq r \right\} \in \mathcal{B}(\mathbb{R}).
\end{align*}
\end{enumerate}
(For a proof of existence of such set we refer the reader to \cite[Theorem 1.10]{Rudin:1987}.)
\end{definition}

A \textit{non-negative Borel measure} on $\mathbb{R}$ is a set function $\mu$, defined on $\mathcal{B}(\mathbb{R})$, whose range belongs to $\overline{\mathbb{R}}_{\geq} \coloneqq \{ r \in \overline{\mathbb{R}} : r \geq 0 \}$ and which is $\sigma$-additive.  A \textit{signed Borel measure} on $\mathbb{R}^{n}$ is a $\sigma$-additive set function with domain $\mathcal{B}(\mathbb{R}^{n})$ whose range is a subset of $\overline{\mathbb{R}}$.

We remind the reader that if $A \subseteq \mathbb{R}$ and if $r$ is a non-zero positive real number, then a countable (or finite) collection of sets $\{ B(x_{k}, r_{k}) \}_{k \in \mathbb{N}}$ that cover $A$ (that is $A \subset \bigcup_{k \in \mathbb{N}} B(x_{k}, r_{k})$) with $0 < r_{k} < r$ and $x_{k} \in A$, for each $k \in \mathbb{N}$, is called a \textit{centred $r$-cover} of $A$.

\begin{definition}\label{defn:Hausdorff}
For $A \in \mathcal{B}(\mathbb{R})$ and for any $r > 0$ we define
\begin{align*}
\lambda_{r}(A) \coloneqq \inf \left\{ \sum_{k \in \mathbb{N}} 2r_{k}  : \{ B(x_{k}, r_{k} \}_{k \in \mathbb{N}} \;\text{is a centred $r$-cover of $A$} \right\}.
\end{align*}
We then define the $1$-dimensional Lebesgue measure of $A$ by $\displaystyle{\lambda(A) \coloneqq \lim_{r \to 0} \lambda_{r}(A)}$.
\end{definition}

This is a non-trivial, non-negative, $\sigma$-additivity and translation invariant set function with domain $\mathcal{B}(\mathbb{R})$ which assigns the value $r$ to the set $B(x, r)$, for each $x \in \mathbb{R}$ and every non-negative real number $r$.  For more details on the Lebesgue measure, and for a proof that it is indeed a measure, we refer the reader to either of \cite{Bogachev:2007, Falconer:1990, Rudin:1987}.  Note that, Borel $\sigma$-algebras and Borel measures can be defined for general topological spaces.  Such examples of non-negative Borel measures include the $n$-dimensional Lebesgue Measure \cite{Bogachev:2007, Falconer:1990, Rudin:1987}, the Hausdorff and packing measures \cite{Edgar:1997, Falconer:1990, Falconer:1997, Olsen:1994} and Gibbs measures \cite{Bowen:1975, Falconer:1997}.

Let $f: \mathbb{R} \to \mathbb{R}$ denote a continuous (or more generally a measurable) function such that there exist sets $A_{1}, A_{2} \in \mathcal{B}(\mathbb{R})$ with $\lambda(A_{1}) > 0$ and $\lambda(A_{2}) > 0$, and $f(a_{1}) < 0$ and $f(a_{2}) > 0$, for all $a_{1} \in A_{1}$ and $a_{2} \in A_{2}$.  One can then define a signed Borel measure $\nu$ by
\begin{align*}
\nu(A) \coloneqq \int_{A} f(x) \, d\lambda^{n}(x),
\end{align*}
for each $A \in \mathcal{B}(\mathbb{R}^{n})$.  For the definition of integration of a continuous (or a measurable) function with respect to a measure and the definition of a measurable function, see either \cite{Bogachev:2007, Rudin:1987}.  Other examples of signed Borel measures include the curvature measures defined by Federer \cite{Federer:1959}.

\section{Proof of Theorem \ref{Main-Theo2}}\label{Sec:sectio2}

The proof of Theorem \ref{Main-Theo2} follows from the application of three lemmas, and is motivated by the standard proof of Vitali Theorem (Theorem \ref{Main-Theo1}).

Define the relation $\sim$ on the closed unit interval $[0, 1]$ by $x \sim y$ if and only if $x - y \in \mathbb{Q}$, for all $x, y \in [0, 1]$.  Observe that the relation $\sim$ is an equivalence relation since:
\begin{enumerate}
\item[(a)] We have that $x \sim x$, for all $x \in \mathbb{R}$; this follows since zero is a rational number.
\item[(b)] If $x \sim y$ then $y \sim x$, for some $x, y \in [0, 1]$; this follows since, if $q$ is a rational number, then $-q$ is also a rational number.
\item[(c)] If $x \sim y$ and $y \sim z$, then $x \sim z$, for some $x, y, z \in [0, 1]$; this follows since the sum of two rational numbers is a rational number.
\end{enumerate}
We remind the reader that an \textit{equivalence class} is a complete collection of elements which are equivalent and that the set of all equivalence classes partition the space on which the equivalence relation is defined; in our case this set will be denoted by $[0, 1] / \sim$.

Let $B$ denote a subset of $[0, 1]$, which contains exactly one element from each equivalence class of  $[0, 1]/\!\sim$, where for the equivalence class containing zero we choose the point zero.  The construction of the set $B$ is possible since we have assumed the axiom of choice.  For each $q \in \mathbb{Q} \cap [0, 1)$, define $B_{q} \coloneqq (( B + q) \cap [0,1]) \cup ( (B + (q - 1)) \cap [0, 1])$.  Observe that
\begin{align}\label{lem:4}
(B + q) \cap [0,1] \cap (B + (q - 1)) \cap [0, 1] = \emptyset,
\end{align}
for each $q \in \mathbb{Q} \cap [0, 1)$, since, by definition, $B + q \subseteq [q, q+1)$ and $B+(q-1) \subseteq[q-1, q)$.

\begin{lemma}\label{lem:lem1}
Let $B_{q}$ be defined as above, then
\begin{align}\label{eq:2}
\bigcup_{q \in \mathbb{Q} \cap [0, 1)} B_{q} = [0, 1).
\end{align}
\end{lemma}

\begin{proof}
Since, by definition of the sets $B_{q}$, we have that $B_{q} \subseteq [0, 1)$, for all $q \in \mathbb{Q} \cap [0, 1)$, it follows that the left-hand-side of Equation (\ref{eq:2}) is a subset of $[0, 1)$.

To show the converse, suppose that $z \in [0, 1)$. Since $\sim$ is an equivalence relation, there exists a unique $y \in B$ such that $z \sim y$, which by definition implies that there exists a $q \in \mathbb{Q} \cap (-1, 1)$ such that $z - y = q$.  Hence, $z \in B_{q}$ if $q \geq 0$, else $z \in B_{1 + q}$, and so $z$ belongs to the left-hand-side of Equation (\ref{eq:2}).
\end{proof}

\begin{lemma}\label{lem:lem2}
For $p \neq q \in \mathbb{Q} \cap [0, 1)$, we have that $B_{q} \cap B_{p} = \emptyset$.
\end{lemma}

\begin{proof}
We prove this by way of contradiction.  Suppose that there exist $p \neq q \in \mathbb{Q} \cap [0, 1)$, such that $B_{q} \cap B_{p} \neq \emptyset$.  Let $z$ be an element of this intersection.  By Equation (\ref{lem:4}) we have either
\begin{enumerate}
\item[(a)] $z \in (B+q) \cap [0,1]$ and $z \in (B+p) \cap [0, 1]$,
\item[(b)] $z \in (B+(q-1)) \cap [0,1]$ and $z \in (B+(p-1)) \cap [0, 1]$,
\item[(c)] $z \in (B+q) \cap [0,1]$ and $z \in (B+(p-1)) \cap [0, 1]$, or
\item[(d)] $z \in (B+(q-1)) \cap [0,1]$ and $z \in (B+p) \cap [0, 1]$.
\end{enumerate}
Cases (a) and (b), and Case (c) and (d) are symmetric, and so, it suffices to consider Cases (a) and (c) only.  

Suppose that Case (a) occurs.  Then there exist $z_{1}, z_{2} \in B$, such that $z_{1} + q = z = z_{2} + p$.  Hence, $z_{1} - z_{2} \in \mathbb{Q}$ and so, $z_{1} = z_{2}$ by definition of $B$.  Thus $p = q$, which is a contradiction, since $p$ and $q$ were chosen such that $q \neq p$.

Suppose that Case (c) occurs.  Then there exist $z_{1}, z_{2} \in B$, such that $z_{1} + q = z = z_{2} + p - 1$.  Hence, $z_{1} - z_{2} \in \mathbb{Q}$ and so, $z_{1} = z_{2}$ by definition of $B$.  Thus $p - q = 1$, which cannot happen if $p, q \in \mathbb{Q} \cap [0, 1)$.
\end{proof}

\begin{lemma}\label{lem:lem3}
For each non-trivial and atom-less set function $\mathcal{L}: \mathcal{P}(\mathbb{R}) \to \overline{\mathbb{R}}$ which is $\sigma$-additive and translation invariant, we have that $\mathcal{L}(B) = \mathcal{L}(B + q) = \mathcal{L}(B_{q})$, for all $q \in \mathbb{Q} \cap [0, 1)$.
\end{lemma}

\begin{proof}
For each $q \in \mathbb{Q} \cap [0, 1)$, we have that
\begin{align*}
\mathcal{L}(B) &= \mathcal{L}(B + q)\\
&= \mathcal{L} \, \big(((B + q) \cap [0, 1]) \, \textstyle{\bigcup\hspace{-2.5mm}\cdot} \;\, ((B + q) \cap (1, 2))\big)\\
&= \mathcal{L}((B + q) \cap [0, 1]) + \mathcal{L}((B + q) \cap (1, 2))\\
&= \mathcal{L}((B + q) \cap [0, 1]) + \mathcal{L}((B + (q - 1)) \cap (0, 1))\\
&= \mathcal{L} \, \big(((B + q) \cap [0, 1])  \, \textstyle{\bigcup\hspace{-2.5mm}\cdot} \;\, ((B + (q - 1)) \cap [0, 1])\big) = \mathcal{L}(B_{q}),
\end{align*}
where the symbol $\bigcup\hspace{-2.5mm}\cdot\;$ denotes the disjoint union of two sets.  The first and fourth equalities are due to translation invariance; the second follows since $B \subset [0, 1]$ and $q \in \mathbb{Q} \cap [0, 1)$; the third and fifty equalities are due to $\sigma$-additivity and Equation (\ref{lem:4}); the last equality follows from the definition of $B_{q}$.
\end{proof}

\begin{proof}[Proof of Theorem \ref{Main-Theo2}]
Assume the contrary, that is, there exists a non-trivial and atom-less set function $\mathcal{L}: \mathcal{P}(\mathbb{R}) \to \overline{\mathbb{R}}$ which is $\sigma$-additive and translation invariant with $\mathcal{L}([0, 1]) = 1$.  Firstly observe that $\mathcal{L}([0, 1)) = 1$, since $\mathcal{L}$ is atom-less and $\sigma$-additive.  Secondly, observe that
\begin{align*}
\mathcal{L}([0, 1)) = \mathcal{L} \,\bigg( \bigcup_{q \in \mathbb{Q} \cap [0, 1)}\hspace{-6.4mm}\cdot \hspace{5mm}B_{q} \bigg) = \sum_{q \in \mathbb{Q} \cap [0,1)} \mathcal{L} (B_{q}) = \sum_{q \in \mathbb{Q} \cap [0,1)}\mathcal{L}(B),
\end{align*}
where the first equality is due to Lemma \ref{lem:lem1}; the second equality follows from the fact that $\mathcal{L}$ is $\sigma$-additive and Lemma \ref{lem:lem2}; the final equality follows from Lemma \ref{lem:lem3}.  However, this is a contradiction, since if $\mathcal{L}(B) = 0$, then the last summand is zero; and if $\mathcal{L}(B) \neq 0$, then last summand is is either positive or negative infinity.
\end{proof}

\section*{Acknowledgements}

The author would like to thank members of the University of St Andrews and Bernd Stratmann (Universt\"at Bremen) for encouraging him to submit this note for publication. The author would also like to thank the referee for his comments.

\end{document}